\documentclass[11pt]{article}%
\usepackage[letterpaper, portrait, margin=1in, headheight=0pt]{geometry}
\usepackage{mathtools}
\usepackage{enumerate}
\usepackage{amsmath,amsthm,amssymb,amsfonts}
\usepackage{enumitem}
\usepackage{mathrsfs}
\usepackage{tikz-cd}
\usepackage{multicol}
\usepackage[title]{appendix}
\usetikzlibrary{quotes,angles}

\DeclareMathOperator{\Hom}{Hom}
\DeclareMathOperator{\Bil}{Bil}

\DeclareMathOperator{\Inf}{Inf}
\DeclareMathOperator{\Tra}{Tra}
\DeclareMathOperator{\Res}{Res}
\DeclareMathOperator{\ima}{Im}

\newcommand{\F}{\mathbb{F}}
\newcommand{\E}{\varepsilon}
\newcommand{\U}{\omega}

\newcommand{\vp}{\varphi}
\newcommand{\LL}{L}

\newcommand{\cZ}{\mathcal{Z}}
\newcommand{\cB}{\mathcal{B}}
\newcommand{\cH}{\mathcal{H}}

\newcommand{\vv}{_{\vdash}}
\newcommand{\dd}{_{\dashv}}
\newcommand{\FR}{F\lozenge R + R\lozenge F}

\newtheorem{thm}{Theorem}[section]
\newtheorem{lem}[thm]{Lemma}

\newtheorem{cor}[thm]{Corollary}

\theoremstyle{definition}

\newtheorem{ex}[subsection]{Example}

\theoremstyle{remark}

\title{Multipliers and Unicentral Diassociative Algebras}
\author{Erik Mainellis}
\date{}

\begin{document}

\maketitle

\begin{abstract}
    The work of this paper can be divided into three primary undertakings. We first prove that covers of diassociative algebras are unique. Second, we show that the multiplier of a diassociative algebra is characterized by the second cohomology group with coefficients in the field. Third, we establish criteria for when the center of a cover maps onto the center of the algebra. Along the way, we obtain a collection of exact sequences, characterizations, and a brief theory of unicentral diassociative algebras and stem extensions. This paper is part of an ongoing project to advance extension theory in the context of several Loday algebras.
\end{abstract}

\section{Introduction}
The objective of this paper is to develop diassociative analogues of Chapters 1, 3, and 4 from Peggy Batten's 1993 dissertation \cite{batten}. Therein, the author obtained results concerning the multipliers and covers of Lie algebras, as discussed in \cite{mainellis batten}. These chapters were generalized to Leibniz algebras in \cite{rogers} and \cite{mainellis batten}. This research highlights an extension-theoretic crossroads between multipliers and covers, unicentral algebras, and the second cohomology group. Some of the results in the present paper hold by the same logic as their Lie or Leibniz versions, and so we will sometimes refer to those proofs rather than rewriting them.

Let $\F$ be a field. A \textit{diassociative algebra} $L$ is an $\F$-vector space equipped with two associative bilinear products $\dashv$ and $\vdash$ which satisfy
\begin{align*}
    x\dashv (y\dashv z) = x\dashv (y\vdash z),\\
    (x\vdash y)\dashv z = x\vdash (y\dashv z),\\
    (x\dashv y)\vdash z = (x\vdash y)\vdash z.
\end{align*}
for all $x,y,z\in L$. Let $A$ and $B$ be diassociative algebras. An \textit{extension} of $A$ by $B$ is a short exact sequence $0\xrightarrow{} A\xrightarrow{\sigma} L\xrightarrow{\pi} B\xrightarrow{} 0$ for which $\sigma$ and $\pi$ are homomorphisms and $L$ is a diassociative algebra. In general, one may assume that $\sigma$ is the identity map, and we make this assumption throughout the paper. An extension is called \textit{central} if $A\subseteq Z(L)$. A \textit{section} of an extension is a linear map $\mu:B\xrightarrow{} L$ such that $\pi\circ \mu = \text{id}_B$. We denote by $A\lozenge B$ the algebra $A\dashv B + A\vdash B$.

\section{Existence of Universal Elements and Unique Covers}\label{batten 1}
Let $L$ be a diassociative algebra. A pair of diassociative algebras $(K,M)$ is a \textit{defining pair} for $L$ if $L\cong K/M$ and $M\subseteq Z(K)\cap K'$. Such a pair is a \textit{maximal defining pair} if the dimension of $K$ is maximal. In this case, $K$ is called a \textit{cover} of $L$ and $M$ is called the \textit{multiplier} of $L$, denoted by $M(L)$. The multiplier is abelian and thus unique via dimension. The main objective of this section is to prove the uniqueness of covers in the diassociative setting. The initial dimension bounds are notably different from the Lie and Leibniz cases, as there are simply more possible multiplications for which to account.

\begin{lem}\label{dias batten 1.1}
For any diassociative algebra $K$, if $\dim(K/Z(K)) = n$, then $\dim(K')\leq 2n^2$.
\end{lem}

\begin{proof}
Let $\{\overline{x_1},\overline{x_2},\dots, \overline{x_n}\}$ be a basis for $K/Z(K)$. Then $\{x_i\dashv x_j, x_i\vdash x_j~|~ 1\leq i,j\leq n\}$ is a generating set for $K'$. Thus $\dim(K')\leq 2n^2$.
\end{proof}

\begin{lem}\label{dias batten 1.2}
Let $L$ be a finite-dimensional diassociative algebra with $\dim L = n$ and let $K$ be the first term in a defining pair for $L$. Then $\dim K\leq n(2n+1)$.
\end{lem}

\begin{proof}
We know that $\dim(K/Z(K)) \leq \dim(K/M) = \dim L = n$ since $M\subseteq Z(K)$. Therefore, $\dim M\leq \dim(K')\leq 2n^2$ via Lemma \ref{dias batten 1.1} since $M\subseteq K'$. We thus have $\dim K = \dim L + \dim M \leq n+2n^2 = n(2n+1)$.
\end{proof}

These facts ensure that the members of any defining pair for a finite-dimensional diassociative algebra $L$ have bounded dimension. Example \ref{diassociative dimension bound obtained} illustrates that the highest possible dimension bounds of Lemmas \ref{dias batten 1.1} and \ref{dias batten 1.2} can always be obtained.

\begin{ex}\label{diassociative dimension bound obtained}
Let $L$ be the $n$-dimensional abelian diassociative algebra with basis $\{x_i\}_{i=1,\dots,n}$ and let $M$ be the $2n^2$-dimensional abelian diassociative algebra with basis $\{m_{ij}, s_{ij}\}_{i,j = 1,\dots n}$. Let $K$ denote the vector space $M\oplus L$ with only nonzero multiplications given by $x_i\dashv x_j = m_{ij}$ and $x_i\vdash x_j = s_{ij}$ for $i,j=1,\dots, n$. Then $K$ is a diassociative algebra of dimension $n + 2n^2$ and $M = Z(K) = K'$. Clearly $K$ is a cover of $L$ and $M$ is the multiplier since we have maximal possible dimension. Noting that $L=K/Z(K)$, we also obtain $\dim K' = 2n^2$.
\end{ex}

Let $C(L)$ denote the set of all pairs $(J,\lambda)$ such that $\lambda:J\xrightarrow{} L$ is a surjective homomorphism and $\ker \lambda\subseteq J'\cap Z(J)$. An element $(T,\tau)\in C(L)$ is called a \textit{universal element} in $C(L)$ if, for any $(J,\lambda)\in C(L)$, there exists a homomorphism $\beta:T\xrightarrow{} J$ such that the diagram \[\begin{tikzcd}
T\arrow[r, "\tau"]\arrow[d,swap, "\beta"]& L\\
J \arrow[ur, swap, "\lambda"]
\end{tikzcd}\] commutes, i.e. such that $\lambda\circ \beta = \tau$.

Defining pairs for $L$ correspond to elements of $C(L)$ in a natural way. Indeed, any $(K,\lambda)\in C(L)$ gives rise to a defining pair $(K,\ker \lambda)$. Conversely, any defining pair $(K,M)$ yields a surjective homomorphism $\lambda:K\xrightarrow{} L$ such that $\ker \lambda = M\subseteq Z(K)\cap K'$, and thus $(K,\lambda)\in C(L)$. We will show that a pair $(T,\tau)\in C(L)$ is a universal element if and only if $T$ is a cover.

\begin{lem}\label{dias batten 1.3}
Let $K$ be a finite dimensional diassociative algebra. Then $Z(K)\cap K'$ is contained in every maximal subalgebra of $K$.
\end{lem}

\begin{proof}
Let $M$ be a maximal subalgebra of $K$ and let $A= Z(K)\cap K'$. Then $A+M$ is also subalgebra of $K$, which implies that $A+M= K$ or $M$. Suppose $A+M=K$. Then \begin{align*}
    K' &= K\dashv K + K\vdash K \\ &= A \lozenge A + A\lozenge M + M\lozenge A + M\lozenge M \\ &= M\dashv M + M\vdash M \\&= M' \subseteq M.
\end{align*} Therefore $Z(K)\cap K'\subseteq M$, a contradiction.
\end{proof}

\begin{lem}\label{dias batten 1.4}
Let $(J,\lambda)\in C(L)$ and $\mu:K\xrightarrow{} L$ be a surjective homomorphism. Suppose there is a homomorphism $\beta:K\xrightarrow{} J$ such that the diagram \[\begin{tikzcd}
K\arrow[r, "\mu"]\arrow[d,swap, "\beta"]& L\\
J \arrow[ur, swap, "\lambda"]
\end{tikzcd}\] commutes, i.e. such that $\lambda\circ \beta = \mu$. Then $\beta$ is surjective.
\end{lem}

\begin{proof}
Let $j\in J$ and $\lambda(j) = \mu(k)$ for some $k\in K$. Then the equality $\mu(k) = \lambda\circ \beta(k) = \lambda(j)$ yields an element $\beta(k)-j\in \ker \lambda$ and so $J = \ker \lambda + \ima \beta$. By assumption, $\ker \lambda\subseteq Z(J)\cap J'$, where $Z(J)\cap J'$ is contained in every maximal subalgebra of $J$ by Lemma \ref{dias batten 1.3}. Suppose that $\ima \beta\neq J$. Then $\ima \beta$ is contained in some maximal subalgebra $M$ of $J$ which is not equal to $J$, and so $\ima \beta + (Z(J)\cap J')\subseteq M$. But this implies that $M=J$, a contradiction. Thus $\ima \beta = J$.
\end{proof}

Suppose there is an element $(K,\eta)\in C(L)$ such that, for all $(J,\lambda)\in C(L)$, there exists a homomorphism $\rho:K\xrightarrow{} J$ which satisfies $\lambda\circ \rho = \eta$. Then Lemma \ref{dias batten 1.4} implies that $\rho$ is surjective, and so $\dim J\leq \dim K$. Hence $K$ is a cover of $L$ since its dimension is maximal. Moreover, any other cover of $L$ has the same dimension as $K$ and is the homomorphic image of $K$, and so must be isomorphic to $K$. We have thus shown the following statement.

\begin{lem}
If there exists a universal element $(K,\eta)\in C(L)$, then all covers of $L$ are isomorphic.
\end{lem}

It remains to show that universal elements exist in the diassociative setting. We begin by fixing a free presentation $0\xrightarrow{} R\xrightarrow{} F\xrightarrow{\pi} L\xrightarrow{} 0$ of $L$ and assigning \[B= \frac{R}{\FR} \hspace{2cm} C = \frac{F}{\FR} \hspace{2cm} D=\frac{F'\cap R}{\FR}\] for ease of notation. Then $C\lozenge B+B\lozenge C = 0$ and $D$ is a central ideal in $C$. Thus $\pi$ induces a homomorphism $\overline{\pi}:C\xrightarrow{} L$ with kernel $B$. We will show that there is a central ideal \[E = \frac{S}{\FR}\] of $C$, complementary to $D$ in $B$, such that $(C/E,\pi_S)\in C(L)$ is a universal element, where $\pi_S$ is induced by $\overline{\pi}$. Consider $(J,\lambda)\in C(L)$. By the universal property of our free diassociative algebra $F$, there exists a homomorphism $\sigma:F\xrightarrow{} J$ such that the diagram \[\begin{tikzcd}
F\arrow[r, "\pi"]\arrow[d,swap, "\sigma"]& L\\
J \arrow[ur, swap, "\lambda"]
\end{tikzcd}\] commutes, i.e. such that $\pi = \lambda\circ \sigma$, as seen in the Lie and Leibniz cases. The following lemma yields a commutative diagram \[\begin{tikzcd}
C\arrow[r, "\overline{\pi}"]\arrow[d,swap, "\overline{\sigma}"]& L\\
J \arrow[ur, swap, "\lambda"]
\end{tikzcd}\] where $\overline{\sigma}$ is induced by $\sigma$.

\begin{lem}\label{dias batten 1.6}
Let $x\in F$. Then $x\in R$ if and only if $\sigma(x)\in \ker \lambda$. Also, $\FR\subseteq \ker \sigma$, and thus $\sigma$ induces a homomorphism $\overline{\sigma}:C\xrightarrow{} J$ which is surjective and satisfies $\lambda\circ \overline{\sigma} = \overline{\pi}$.
\end{lem}

\begin{proof}
Let $x\in F$. If $x\in R$, then $x\in \ker \pi$, which implies that $0=\pi(x) = \lambda\circ \sigma(x)$. Therefore $\sigma(x)\in \ker \lambda$. Conversely, if $\sigma(x)\in \ker \lambda$, then $0 = \lambda\circ \sigma(x) = \pi(x)$, which implies that $x$ is an element of $\ker \pi = R$. Now consider $r\dashv f\in \FR$. Then $\sigma(r\dashv f) = \sigma(r)\dashv \sigma(f) = 0$ since $\sigma(r)\in \ker \lambda\subseteq Z(J)\cap J'$ and $\sigma(f)\in J$. The cases of $r\vdash f$, $f\dashv r$, and $f\vdash r$ are similar, and so we have our homomorphism $\overline{\sigma}:C\xrightarrow{} J$, induced by $\sigma$, which is surjective since $\sigma$ is surjective. One computes \begin{align*}
    \lambda\circ \overline{\sigma}(f+(\FR)) &= \lambda\circ \sigma(f) \\ &= \pi(f) \\ &= \overline{\pi}(f+(\FR))
\end{align*} and thus $\lambda\circ \overline{\sigma} = \overline{\pi}$.
\end{proof}

\begin{lem}\label{batten combo}
$\overline{\sigma}(B) = \ker \lambda = \overline{\sigma}(D)$, from which it follows $B = D+\ker \overline{\sigma}$.
\end{lem}

\begin{proof}
This lemma combines the results of Lemmas 1.7, 1.8, and 1.9 in \cite{batten}, which follow similarly to the Lie case.
\end{proof}

\begin{lem}\label{dias batten 1.5}
$(C/E,B/E)$ is a defining pair for $L$ where $E$ is a central ideal in $C$ that is complementary to $D$ in $B$.
\end{lem}

\begin{proof}
We first compute \[\frac{C/E}{B/E}\cong C/B\cong F/R\cong L\] and thus the first axiom of defining pairs is satisfied. Next, we know that $B\subseteq Z(C)$, and so \[B/E\subseteq Z(C/E).\] Finally, \[D= \frac{F'\cap R}{\FR} \subseteq \frac{F'}{\FR} \cong \Big(\frac{F}{\FR}\Big)' = C'\] implies that \[B/E\cong \frac{D\oplus E}{E} \subseteq \frac{C'+E}{E} \cong (C/E)'.\] Therefore $B/E\subseteq Z(C/E)\cap (C/E)'$.
\end{proof}

Lemma \ref{batten combo} allows us to choose a subspace $E$, complementary to $D$ in $B$, which is contaied in $\ker \overline{\sigma}$. Thus, given an element $(J,\lambda)\in C(L)$, our $\overline{\sigma}$ induces a homomorphism $\sigma_S:C/E\xrightarrow{} J$ such that the diagram \[\begin{tikzcd}
C/E\arrow[r, "\pi_S"]\arrow[d,swap, "\sigma_S"]& L\\
J \arrow[ur, swap, "\lambda"]
\end{tikzcd}\] commutes, i.e. such that $\lambda\circ \sigma_S = \pi_S$. Furthermore, by Lemma \ref{dias batten 1.5}, $(C/E,B/E)$ is a defining pair for $L$. Specializing this discussion to when $J$ is a cover of $L$, we now prove that $C/E$ is a cover and thereby obtain a characterization of the multiplier in terms of the free presentation.

\begin{lem}
Given a cover $K$ of $L$, the corresponding $C/E$ is also a cover of $L$ and the multiplier of $L$ is \[M(L) = \frac{F'\cap R}{\FR}.\]
\end{lem}

\begin{proof}
If $K$ is a cover of $L$, then $\dim K\geq \dim(C/E)$ since $C/E$ is the first member of a defining pair for $L$. Since $K$ is the homomorphic image of $C/E$, we also have $\dim K\leq \dim(C/E)$ and thus $\dim K = \dim(C/E)$. This means $C/E$ is a cover of $L$. Finally, since $C/B\cong L$ and \[B/E\cong \frac{F'\cap R}{\FR},\] we have the desired expression for $M(L)$.
\end{proof}

Since $E\subseteq \ker \overline{\sigma}$, $E$ necessarily depends on $\overline{\sigma}$ and thus on $J$. As in \cite{rogers} and \cite{batten}, we will show that there is a single $C/E$ which works for all $(J,\lambda)\in C(L)$, i.e. a \textit{universal} element of the form $(C/E,\pi_S)\in C(L)$. It suffices to show that all $C/E$'s are isomorphic. To this end, we first state the following cancellation lemma, which holds by the same logic as its Lie analogue. We will then specialize $F$ so that said lemma can be applied.

\begin{lem}\label{dias batten 1.11}
Let $L= B\oplus D = B_1\oplus D_1$. If $B\cong B_1$ and $B$ is finite-dimensional, then $D\cong D_1$.
\end{lem}

Denote $n=\dim L$. Noting that $L$ is the homomorphic image of $F$, let $F$ be generated by $n$ elements. Then \[E\cong B/D \cong \frac{R}{F'\cap R}\cong \frac{F'+R}{F'} \subseteq F/F'\] where $F/F'$ is abelian and generated by $n$ elements, and thus $E$ is finite dimensional. Next, consider $(J,\lambda)\in C(L)$. As above, one obtains a central ideal $E_1$ in $C$, complementary to $D$ in $B$, and a homomorphism $\overline{\sigma}:C\xrightarrow{} J$ such that $E_1\subseteq \ker \overline{\sigma}$ and $\lambda\circ \overline{\sigma} = \overline{\pi}$. Since $E\cap C'$ and $E_1\cap C'$ are both zero, we may extend $D$ in two different ways. First, extend $D$ to a space $G$ such that $C=E\oplus G$. Second, extend $D$ to a space $G_1$ such that $C=E_1\oplus G_1$. Since both $E$ and $E_1$ have the same finite dimension and are abelian, we know $E\cong E_1$. By Lemma \ref{dias batten 1.11}, $G\cong G_1$. Thus $C/E\cong C/E_1$.

In conclusion, we have shown that, given a free presentation $0\xrightarrow{} R\xrightarrow{} F\xrightarrow{\pi} L\xrightarrow{} 0$ and any $(J,\lambda)\in C(L)$, one can choose a subspace $E$ in $C$ and induce a homomorphism $\sigma_S:C/E\xrightarrow{} J$ such that $\lambda\circ \sigma_S = \pi_S$. By Lemma \ref{dias batten 1.4}, $\sigma_S$ is surjective. Thus $(C/E,\pi_S)$ is a universal element of $C(L)$ and $C/E$ is a cover of $L$. Furthermore, each cover $K$ of $L$ is the homomorphic image of $C/E$ and has the same dimension, and so every cover is isomorphic to $C/E$. Finally, for any $(J,\lambda)\in C(L)$ and a cover $K$ of $L$, there exists a homomorphism $\beta:K\xrightarrow{} J$ such that $\lambda\circ \beta = \tau$ where $(K,\tau)\in C(L)$.
Thus $K$ is a cover of $L$ if and only if $(K,\tau)$ is a universal element in $C(L)$.

\begin{thm}\label{dias batten 1.12}
Let $L$ be a finite-dimensional diassociative algebra and let $0\xrightarrow{} R\xrightarrow{} F\xrightarrow{} L\xrightarrow{} 0$ be a free presentation of $L$. Let \[B=\frac{R}{\FR} ~~~~~~ C=\frac{F}{\FR} ~~~~~~ D=\frac{F'\cap R}{\FR}\] Then \begin{enumerate}
    \item All covers of $L$ are isomorphic and have the form $C/E$ where $E$ is the complement to $D$ in $B$.
    \item The multiplier $M(L)$ of $L$ is $D\cong B/E$.
    \item The universal elements in $C(L)$ are the elements $(K,\lambda)$ where $K$ is a cover of $L$.
\end{enumerate}
\end{thm}

\section{Multipliers and Cohomology}\label{batten 3}
The objective of this section is to characterize the multiplier of a diassociative algebra in terms of the second cohomology group. Given a pair of diassociative algebras $A$ and $B$, consider a central extension $0\xrightarrow{} A\xrightarrow{} L\xrightarrow{} B\xrightarrow{} 0$ of $A$ by $B$ and section $\mu:B\xrightarrow{} L$. Define a pair of bilinear forms $(f\dd,f\vv)$ by $f\dd(i,j) = \mu(i)\dashv\mu(j) - \mu(i\dashv j)$ and $f\vv(i,j) = \mu(i)\vdash\mu(j) - \mu(i\vdash j)$ for $i,j\in B$. Then the images of $f\dd$ and $f\vv$ fall in $A$ by exactness. By our work on factor systems in \cite{mainellis}, $(f\dd,f\vv)$ is a 2-cocycle of diassociative algebras, meaning that these maps satisfy the axioms of central factor systems, i.e. \begin{enumerate}
    \item[C1.] $f\dd(i,j\dashv k) = f\dd(i,j\vdash k)$,
    \item[C2.] $f\dd(i\vdash j,k) = f\vv(i,j\dashv k)$,
    \item[C3.] $f\vv(i\dashv j,k) = f\vv(i\vdash j,k)$,
    \item[C4.] $f\dd(i,j\dashv k) = f\dd(i\dashv j,k)$,
    \item[C5.] $f\vv(i,j\vdash k) = f\vv(i\vdash j,k)$
\end{enumerate} for $i,j,k\in B$. Let $\cZ^2(B,A)$ denote the set of all 2-cocycles and $\cB^2(B,A)$ denote the set of all 2-coboundaries, i.e. 2-cocycles $(f\dd,f\vv)$ such that $f\dd(i,j) = -\E(i\dashv j)$ and $f\vv(i,j) = -\E(i\vdash j)$ for some linear transformation $\E:B\xrightarrow{} A$. We next note that any elements $(f\dd,f\vv)$ and $(g\dd,g\vv)$ in $\cZ^2(B,A)$ belong to equivalent extensions if and only if their corresponding bilinear forms differ by a coboundary, i.e. if there is a linear map $\E:B\xrightarrow{} A$ such that $f\dd(i,j) - g\dd(i,j) = -\E(i\dashv j)$ and $f\vv(i,j) - g\vv(i,j) = -\E(i\vdash j)$ for all $i,j\in B$. Therefore, extensions of $A$ by $B$ are equivalent if and only if they give rise to the same element of $\cH^2(B,A) = \cZ^2(B,A)/\cB^2(B,A)$, the \textit{second cohomology group} of $B$ with coefficients in $A$. The work in \cite{mainellis} also guarantees that each element \[\overline{(f\dd,f\vv)}\in \cH^2(B,A)\] gives rise to a central extension $0\xrightarrow{} A\xrightarrow{} L\xrightarrow{} B\xrightarrow{} 0$ with section $\mu$ such that $f\dd(i,j) = \mu(i)\dashv \mu(j) - \mu(i\dashv j)$ and $f\vv(i,j) = \mu(i)\vdash \mu(j) - \mu(i\vdash j)$.

\subsection{Hochschild-Serre Spectral Sequence}
The remainder of this paper relies on the exactness of the following Hochschild-Serre type spectral sequence of low dimension. Let $H$ be a central ideal of a diassociative algebra $L$ and consider the natural central extension $0\xrightarrow{}H\xrightarrow{} L \xrightarrow{\beta}L/H\xrightarrow{} 0$ with section $\mu$ of $\beta$. Let $A$ be a central $L$-module.

\begin{thm}
The sequence \[0\xrightarrow{} \Hom(L/H,A)\xrightarrow{\Inf_1} \Hom(L,A)\xrightarrow{\Res} \Hom(H,A)\xrightarrow{\Tra} \cH^2(L/H,A)\xrightarrow{\Inf_2} \cH^2(L,A)\] is exact.
\end{thm}

We first define the maps in the sequence and verify that they make sense. For any homomorphism $\chi:L/H\xrightarrow{} A$, define $\Inf_1:\Hom(L/H,A)\xrightarrow{} \Hom(L,A)$ by $\Inf_1(\chi) = \chi\circ \beta$. Next, for $\pi\in \Hom(L,A)$, define $\Res:\Hom(L,A)\xrightarrow{} \Hom(H,A)$ by $\Res(\pi) = \pi\circ \iota$ where $\iota:H\xrightarrow{} L$ is the inclusion map. It is readily verified that $\Inf_1$ and $\Res$ are well-defined and linear. To define the transgression map, let $f\dd:L/H\times L/H\xrightarrow{} H$ and $f\vv:L/H\times L/H\xrightarrow{} H$ be defined by $f\dd(\overline{x},\overline{y}) = \mu(\overline{x})\dashv \mu(\overline{y}) - \mu(\overline{x}\dashv \overline{y})$ and $f\vv(\overline{x},\overline{y}) = \mu(\overline{x})\vdash \mu(\overline{y}) - \mu(\overline{x}\vdash \overline{y})$ for $x,y\in L$. Consider $\chi\in \Hom(H,A)$. Then $(\chi\circ f\dd, \chi\circ f\vv)\in \cZ^2(L/H,A)$ since $\chi$ is a homomorphism. Given another section $\nu$ of $\beta$, define a pair $(g\dd,g\vv)$ of bilinear forms by $g\dd(\overline{x}, \overline{y}) = \nu(\overline{x})\dashv \nu(\overline{y}) - \nu(\overline{x}\dashv \overline{y})$ and $g\vv(\overline{x}, \overline{y}) = \nu(\overline{x})\vdash \nu(\overline{y}) - \nu(\overline{x}\vdash \overline{y})$ for $x,y\in L$. Then $(f\dd,f\vv)$ and $(g\dd,g\vv)$ are cohomologous in $\cH^2(L/H,H)$, which implies that there exists a linear transformation $\E:L/H\xrightarrow{} H$ such that $f\dd(\overline{x}, \overline{y}) - g\dd(\overline{x}, \overline{y}) = -\E(\overline{x}\dashv \overline{y})$ and $f\vv(\overline{x}, \overline{y}) - g\vv(\overline{x}, \overline{y}) = -\E(\overline{x}\vdash \overline{y})$. Therefore $\chi\circ \E:L/H\xrightarrow{} A$ is a linear map by which $(\chi\circ f\dd, \chi\circ f\vv)$ and $(\chi\circ g\dd, \chi\circ g\vv)$ differ. In other words, $(\chi\circ f\dd, \chi\circ f\vv)$ and $(\chi\circ g\dd, \chi\circ g\vv)$ are cohomologous in $\cH^2(L/H,A)$, and so we define \[\Tra(\chi) = \overline{(\chi\circ f\dd, \chi\circ f\vv)}.\] It is straightforward to verify that $\Tra$ is linear.

Finally, we define the second inflation map $\Inf_2:\cH^2(L/H,A)\xrightarrow{} \cH^2(L,A)$ by \[\Inf_2((f\dd,f\vv) + \cB^2(L/H,A)) = (f\dd',f\vv') + \cB^2(L,A)\] where $f\dd'(x,y) = f\dd(\beta(x),\beta(y))$ and $f\vv'(x,y) = f\vv(\beta(x),\beta(y))$ for $(f\dd,f\vv)\in \cZ^2(L/H,A)$ and $x,y\in L$. It is straightforward to verify that $\Inf_2$ is linear. To check that $\Inf_2$ maps cocycles to cocycles, one begins by computing \begin{align*}
    f\dd'(x,y\dashv z) &= f\dd(\beta(x),\beta(y\dashv z)) \\ &= f\dd(\beta(x),\beta(y)\dashv \beta(z)) \\ &= f\dd(\beta(x),\beta(y)\vdash \beta(z)) \\ &= f\dd(\beta(x),\beta(y\vdash z)) \\ &= f\dd'(x,y\vdash z)
\end{align*} for all $x,y,z\in L$, which holds since $(f\dd,f\vv)$ is a 2-cocycle. Thus $(f\dd',f\vv')$ satisfies the first axiom of 2-cocycles. The other axioms hold by similar computations and hence $(f\dd',f\vv')\in \cZ^2(L,A)$. To check that $\Inf_2$ maps coboundaries to coboundaries, suppose $(f\dd,f\vv)\in \cB^2(L/H,A)$. Then there is a linear transformation $\E:L/H\xrightarrow{} A$ such that $f\dd(\overline{x},\overline{y}) = -\E(\overline{x}\dashv \overline{y})$ and $f\vv(\overline{x},\overline{y}) = -\E(\overline{x}\vdash \overline{y})$ for all $x,y\in L$. Here $\beta(x) = x+H = \overline{x}$ for any $x\in L$. One has \begin{align*}
    f\dd'(x,y) &= f\dd(\beta(x),\beta(y)) \\ &= -\E(\beta(x)\dashv \beta(y)) \\ &= -\E\circ\beta(x\dashv y)
\end{align*} and, similarly, $f\vv'(x,y) = -\E\circ\beta(x\vdash y)$. Therefore $(f\dd',f\vv')\in \cB^2(L,A)$.

\begin{proof}
Given our section $\mu$ of $0\xrightarrow{} H\xrightarrow{} L\xrightarrow{\beta}L/H\xrightarrow{} 0$, let $(f\dd,f\vv)\in\cZ^2(L/H,H)$ be the cocycle defined by $f\dd(\overline{x},\overline{y}) = \mu(\overline{x})\dashv \mu(\overline{y}) - \mu(\overline{x}\dashv \overline{y})$ and $f\vv(\overline{x},\overline{y}) = \mu(\overline{x})\vdash \mu(\overline{y}) - \mu(\overline{x}\vdash \overline{y})$ for $x,y\in L$. We first note that $\Inf_1$ is injective by the same logic as in \cite{mainellis batten}. Thus the sequence is exact at $\Hom(L/H,A)$. Exactness at $\Hom(L,A)$ also follows similarly to the Leibniz case.

For exactness at $\Hom(H,A)$, first consider a homomorphism $\chi\in \Hom(L,A)$. Then \begin{align*}
    \chi\circ f\dd(\overline{x},\overline{y}) &= \chi\circ\mu(\overline{x})\dashv \chi\circ\mu(\overline{y}) - \chi\circ\mu(\overline{x}\dashv\overline{y}) \\ &= -\chi\circ\mu(\overline{x}\dashv\overline{y})
\end{align*} and, similarly, $\chi\circ f\vv(\overline{x},\overline{y}) = -\chi\circ\mu(\overline{x}\vdash\overline{y})$. This implies that $(\chi\circ f\dd,\chi\circ f\vv)\in \cB^2(L/H,A)$. Thus \[\Tra(\Res(\chi)) = \Tra(\chi\circ\iota) = \overline{(\chi\circ\iota\circ f\dd,\chi\circ\iota\circ f\vv)} = 0\] and so $\ima(\Res)\subseteq \ker(\Tra)$. Conversely, suppose there exists a homomorphism $\theta:H\xrightarrow{} A$ such that \[\Tra(\theta) = \overline{(\theta\circ f\dd,\theta\circ f\vv)} = 0,\] i.e. such that $(\theta\circ f\dd,\theta\circ f\vv)\in \cB^2(L/H,A)$. Then there exists a linear transformation $\E:L/H\xrightarrow{}A$ such that $\theta\circ f\dd(\overline{x},\overline{y}) = -\E(\overline{x}\dashv \overline{y})$ and $\theta\circ f\vv(\overline{x},\overline{y}) = -\E(\overline{x}\vdash \overline{y})$. For any $x,y\in L$, we know that $x=\mu(\overline{x}) + h_x$ and $y=\mu(\overline{y}) + h_y$ for some $h_x,h_y\in H$. Thus $x\dashv y = \mu(\overline{x}\dashv \overline{y}) + h_{x\dashv y} = \mu(\overline{x})\dashv \mu(\overline{y})$ and $x\vdash y = \mu(\overline{x}\vdash \overline{y}) + h_{x\vdash y} = \mu(\overline{x})\vdash \mu(\overline{y})$, which implies that \begin{align}\label{dias chi on H}
    \begin{split}\theta(h_{x\dashv y}) = \theta(\mu(\overline{x})\dashv\mu(\overline{y}) - \mu(\overline{x}\dashv \overline{y})) = \theta\circ f\dd(\overline{x},\overline{y}) = -\E(\overline{x}\dashv \overline{y}), \\ \theta(h_{x\vdash y}) = \theta(\mu(\overline{x})\vdash\mu(\overline{y}) - \mu(\overline{x}\vdash \overline{y})) = \theta\circ f\vv(\overline{x},\overline{y}) = -\E(\overline{x}\vdash\overline{y}).\end{split}
    \end{align}
Define a linear map $\sigma:L\xrightarrow{} A$ by $\sigma(x) = \theta(h_x) + \E(\overline{x})$. Then $\sigma(x)\dashv\sigma(y) = 0$ and $\sigma(x)\vdash \sigma(y)=0$ since $\ima\sigma\subseteq A$. By (\ref{dias chi on H}), \begin{align*}
    \sigma(x\dashv y) = \theta(h_{x\dashv y}) + \E(x\dashv y) = 0,\\
    \sigma(x\vdash y) = \theta(h_{x\vdash y}) + \E(x\vdash y) = 0.
\end{align*} Thus $\sigma$ is a homomorphism. Moreover, $\sigma(h) = \theta(h) + \E(\overline{h}) = \theta(h)$ for all $h\in H$, which implies that $\Res(\sigma) = \theta$. Hence $\ker(\Tra)\subseteq \ima(\Res)$ and $\ker(\Tra) = \ima(\Res)$.

For exactness at $\cH^2(L/H,A)$, first consider a map $\chi\in \Hom(H,A)$. Then \[\Tra(\chi) = (\chi\circ f\dd,\chi\circ f\vv) + \cB^2(L/H,A)\] where $(\chi\circ f\dd,\chi\circ f\vv)\in \cZ^2(L/H,A)$. One computes \[\Inf_2((\chi\circ f\dd,\chi\circ f\vv) + \cB^2(L/H,A)) = ((\chi\circ f\dd)',(\chi\circ f\vv)') + \cB^2(L,A)\] where $(\chi\circ f\dd)'(x,y) = \chi\circ f\dd(\overline{x},\overline{y})$ and $(\chi\circ f\vv)'(x,y) = \chi\circ f\vv(\overline{x},\overline{y})$ for $x,y\in L$. To show that $\ima(\Tra)\subseteq \ker(\Inf_2)$, we need to find a linear transformation $\E:L\xrightarrow{} A$ such that $(\chi\circ f\dd)'(x,y) = -\E(x\dashv y)$ and $(\chi\circ f\vv)'(x,y) = -\E(x\vdash y)$. Let $x=\mu(\overline{x}) + h_x$ and $y=\mu(\overline{y}) + h_y$. Again, the equalities $x\dashv y = \mu(\overline{x}\dashv \overline{y}) + h_{x\dashv y} = \mu(\overline{x})\dashv \mu(\overline{y})$ and $x\vdash y = \mu(\overline{x}\vdash \overline{y}) + h_{x\vdash y} = \mu(\overline{x})\vdash \mu(\overline{y})$ yield \begin{align*}
    \chi\circ f\dd(\overline{x},\overline{y}) = \chi(\mu(\overline{x}\dashv\overline{y} - \mu(\overline{x}\dashv\overline{y})) = \chi(h_{x\dashv y}),\\ \chi\circ f\vv(\overline{x},\overline{y}) = \chi(\mu(\overline{x}\vdash\overline{y} - \mu(\overline{x}\vdash\overline{y})) = \chi(h_{x\vdash y}).
\end{align*} Define $\E(x) = -\chi(h_x)$. Then $E$ is linear and \begin{align*}
    \E(x\dashv y) = -\chi(h_{x\dashv y}) = -\chi\circ f\dd(\overline{x},\overline{y}) = -(\chi\circ f\dd)'(x,y),\\ \E(x\vdash y) = -\chi(h_{x\vdash y}) = -\chi\circ f\vv(\overline{x},\overline{y}) = -(\chi\circ f\vv)'(x,y).
\end{align*} This implies that $((\chi\circ f\dd)', (\chi\circ f\vv)') \in \cB^2(L,A)$ and hence $\ima(\Tra)\subseteq \ker(\Inf_2)$.

Conversely, suppose $(g\dd,g\vv)\in \cZ^2(L/H,A)$ is such that \[\overline{(g\dd,g\vv)}\in \ker(\Inf_2).\] Then there exists a linear transformation $\E:L\xrightarrow{} A$ such that $g\dd(\overline{x},\overline{y}) = g\dd'(x,y) = -\E(x\dashv y)$ and $g\vv(\overline{x},\overline{y}) = g\vv'(x,y) = -\E(x\vdash y)$ for all $x,y\in L$. Since $\E$ is linear, $(\E\circ f\dd,\E\circ f\vv)\in \cZ^2(L/H,A)$. As before, $x=\mu(\overline{x}) + h_x$ and $y=\mu(\overline{y}) + h_y$ for some $h_x,h_y\in H$. Therefore $x\dashv y = \mu(\overline{x})\dashv \mu(\overline{y})$ and $x\vdash y = \mu(\overline{x})\vdash\mu(\overline{y})$. Now \begin{align*}
    g\dd'(x,y) &= g\dd(\overline{x},\overline{y}) \\ &= -\E(x\dashv y)\\ &= -\E(\overline{x}\dashv\overline{y}) \\ &= - \E\circ f\dd(\overline{x},\overline{y}) -\E\circ \mu(\overline{x}\dashv\overline{y})
\end{align*} where $\E\circ \mu:L/H\xrightarrow{} A$. Similarly, $g\vv'(x,y) = -\E\circ f\vv(\overline{x},\overline{y}) - \E\circ \mu(\overline{x}\vdash \overline{y})$. Therefore \[\overline{(g\dd,g\vv)} = \overline{(-\E\circ f\dd,-\E\circ f\vv)} = -\Tra(\E)\] which implies that $\ker(\Inf_2)\subseteq \ima(\Tra)$.
\end{proof}

\subsection{Relation of Multipliers and Cohomology}
Let $L$ be a diassociative algebra and let $\F$ be considered as a central $L$-module. We will now use the Hochschild-Serre sequence to prove that $M(L)\cong \cH^2(L,\F)$. We begin by stating the following theorem, which holds similarly to its Leibniz analogue.

\begin{thm}\label{dias if tra surj}
Let $Z$ be a central ideal in $L$. Then $L'\cap Z$ is isomorphic to the image of $\Hom(Z,\F)$ under the transgression map. In particular, if $\Tra$ is surjective, then $L'\cap Z\cong \cH^2(L/Z,\F)$.
\end{thm}

Now consider a free presentation $0\xrightarrow{} R\xrightarrow{} F\xrightarrow{\U} L\xrightarrow{} 0$ of $L$. The sequence \[0\xrightarrow{} \frac{R}{\FR}\xrightarrow{} \frac{F}{\FR}\xrightarrow{} L\xrightarrow{} 0\] is a central extension since all of $R\dashv F$, $R\vdash F$, $F\dashv R$, and $F\vdash R$ are contained in $\FR$.

\begin{lem}\label{dias restriction to R}
Let $0\xrightarrow{} A\xrightarrow{} B\xrightarrow{\phi} C\xrightarrow{} 0$ be a central extension and $\alpha:L\xrightarrow{} C$ be a homomorphism. Then there exists a homomorphism $\beta:F/(\FR)\xrightarrow{} B$ such that \[\begin{tikzcd}
0\arrow[r]& \frac{R}{\FR}\arrow[r] \arrow[d, "\gamma"] & \frac{F}{\FR} \arrow[r] \arrow[d,"\beta"] & \LL\arrow[r] \arrow[d, "\alpha"] &0 \\
0\arrow[r] &A \arrow[r] & B\arrow[r] &C\arrow[r] &0
\end{tikzcd}\] is commutative, where $\gamma$ is the restriction of $\beta$ to $R/(\FR)$.
\end{lem}

\begin{proof}
Since $F$ is free, there exists a homomorphism $\sigma:F\xrightarrow{} B$ such that \[\begin{tikzcd}
 F\arrow[r, "\U"] \arrow[d, "\sigma", swap] & \LL \arrow[d,"\alpha"] \\ B \arrow[r, "\phi"] & C \end{tikzcd}\] is commutative. Let $r\in R\subseteq F$. Then $\U(r) = 0$ since $\ker \U = R$. Therefore $0=\alpha\circ \U(r) = \phi\circ\sigma(r)$ and so $\sigma(R)\subseteq \ker \phi$. We want to show that $\FR\subseteq \ker \sigma$. If $x\in F$ and $r\in R$, then \begin{align*}
     \sigma(x\dashv r) = \sigma(x)\dashv\sigma(r) = 0, && \sigma(x\vdash r) = \sigma(x)\vdash\sigma(r) = 0, \\ \sigma(r\dashv x) = \sigma(r)\dashv\sigma(x) = 0, && \sigma(r\vdash x) = \sigma(r)\vdash\sigma(x) = 0~
 \end{align*} since $\sigma(r)\in \ker \phi = A \subseteq Z(B)$. Now $\sigma$ induces a homomorphism $\beta:F/(\FR)\xrightarrow{} B$. The left diagram commutes since we may take $A\xrightarrow{} B$ to be the inclusion map.
\end{proof}

\begin{lem}\label{dias tra surj}
Let $0\xrightarrow{} R\xrightarrow{} F\xrightarrow{} L\xrightarrow{} 0$ be a free presentation of $L$. Let $A$ be a central $L$-module. Then the transgression map $\Tra:\Hom(R/(\FR),A)\xrightarrow{} \cH^2(L,A)$ associated with $0\xrightarrow{} \frac{R}{\FR}\xrightarrow{} \frac{F}{\FR}\xrightarrow{\phi} L\xrightarrow{} 0$ is surjective.
\end{lem}

\begin{proof}
Consider $\overline{(g\dd,g\vv)}\in \cH^2(L,A)$ and let $0\xrightarrow{} A\xrightarrow{} E\xrightarrow{\vp} L\xrightarrow{} 0$ be a central extension associated with $\overline{(g\dd,g\vv)}$. By the previous lemma, there exists a homomorphism $\theta$ such that \[\begin{tikzcd}
0\arrow[r]& \frac{R}{\FR}\arrow[r] \arrow[d, "\gamma"] & \frac{F}{\FR} \arrow[r, "\phi"] \arrow[d,"\theta"] & \LL\arrow[r] \arrow[d, "\text{id}"] &0 \\
0\arrow[r] &A \arrow[r] & E\arrow[r, "\vp"] &\LL\arrow[r] &0
\end{tikzcd}\] is commutative and $\gamma = \theta|_{R/(\FR)}$. Let $\mu$ be a section of $\phi$. Then $\vp\circ\theta\circ \mu = \phi\circ \mu = \text{id}$ and so $\theta\circ \mu$ is a section of $\vp$. Let $\lambda = \theta\circ \mu$ and define \begin{align*}
    \beta\dd(x,y) = \lambda(x)\dashv\lambda(y) - \lambda(x\dashv y), \\ \beta\vv(x,y) = \lambda(x)\vdash\lambda(y) - \lambda(x\vdash y).
\end{align*} Then $(\beta\dd,\beta\vv)\in \cZ^2(\LL,A)$ and $(\beta\dd,\beta\vv)$ is cohomologous with $(g\dd,g\vv)$ since they are associated with the same extension. One computes \begin{align*}
    \beta\dd(x,y) &= \theta(\mu(x))\dashv\theta(\mu(y)) - \theta(\mu(x\dashv y)) \\ &= \theta(\mu(x)\dashv\mu(y) - \mu(x\dashv y))\\ &= \gamma(\mu(x)\dashv\mu(y) - \mu(x\dashv y)) \\ &= \gamma(f\dd(x,y))
\end{align*} where $f\dd(x,y) = \mu(x)\dashv\mu(y) - \mu(x\dashv y)$ and since $\gamma = \theta|_{R/(\FR)}$. Similarly, one computes $\beta\vv(x,y) = \gamma(f\vv(x,y))$ for $f\vv(x,y) = \mu(x)\vdash\mu(y) - \mu(x\vdash y)$. Thus \[\Tra(\gamma) = \overline{(\gamma\circ f\dd,\gamma\circ f\vv)} = \overline{(\beta\dd,\beta\vv)} = \overline{(g\dd,g\vv)}\] and $\Tra$ is surjective.
\end{proof}

\begin{lem}\label{dias set lemma}
If $C\subseteq A$ and $C\subseteq B$, then $A/C\cap B/C = (A\cap B)/C$.
\end{lem}

\begin{proof}
Follows by the same logic as the Lie analogue \cite{batten}.
\end{proof}

\begin{thm}
Let $L$ be a diassociative algebra over a field $\F$ and $0\xrightarrow{} R\xrightarrow{} F\xrightarrow{} L\xrightarrow{} 0$ be a free presentation of $L$. Then \[\cH^2(L,\F) \cong \frac{F'\cap R}{\FR}.\] In particular, if $L$ is finite-dimensional, then $M(L)\cong \cH^2(L,\F)$.
\end{thm}

\begin{proof}
Let $\overline{R} = \frac{R}{\FR}$ and $\overline{F} = \frac{F}{\FR}$. Then $0\xrightarrow{} \overline{R}\xrightarrow{} \overline{F}\xrightarrow{} \LL\xrightarrow{} 0$ is a central extension. By Lemma \ref{dias tra surj}, $\Tra:\Hom(\overline{R},\F)\xrightarrow{} \cH^2(\LL,\F)$ is surjective. By Theorem \ref{dias if tra surj}, \[\overline{F}'\cap \overline{R} \cong \cH^2(\overline{F}/\overline{R},\F) \cong \cH^2(\LL,\F).\] By Lemma \ref{dias set lemma}, \[\overline{F}'\cap \overline{R} \cong \frac{F'}{\FR} \cap \frac{R}{\FR} = \frac{F'\cap R}{\FR}.\] Therefore, \[M(\LL) = \frac{F'\cap R}{\FR} \cong \cH^2(\LL,\F)\] by the characterization of $M(L)$ from Theorem \ref{dias batten 1.12}.
\end{proof}

\section{Unicentral Algebras}\label{batten 4}
For a diassociative algebra $L$, let $Z^*(L)$ denote the intersection of all images $\U(Z(E))$ such that $0\xrightarrow{} \ker \U\xrightarrow{} E\xrightarrow{\U} L\xrightarrow{} 0$ is a central extension of $L$. It is easy to see that $Z^*(L)\subseteq Z(L)$. We say that a diassociative algebra $L$ is \textit{unicentral} if $Z(L) = Z^*(L)$. The objective of this section is to develop criteria for when the center of any cover of $L$ maps onto the center of $L$. One of these criteria will take the form of when $Z(L)\subseteq Z^*(L)$, i.e. of when the algebra is unicentral.

\subsection{More Sequences}
We must first extend our Hochschild-Serre sequence by a $\delta$ map analogous to that of the Leibniz case. Given a central ideal $Z$ in $L$, consider the central extension $0\xrightarrow{} Z\xrightarrow{} L\xrightarrow{} L/Z\xrightarrow{} 0$. To define $\delta$, consider a 2-cocycle $(f\dd',f\vv')\in \cZ^2(L,\F)$ and define four bilinear forms \begin{align*}
    f\dd'':L/L'\times Z\xrightarrow{} \F, && f\vv'':L/L'\times Z\xrightarrow{} \F, \\ g\dd'':Z\times L/L'\xrightarrow{} \F, && g\vv'':Z\times L/L'\xrightarrow{} \F~
\end{align*} by \begin{align*}
    f\dd''(x+L',z) = f\dd'(x,z), && f\vv''(x+L',z) = f\vv'(x,z), \\
    g\dd''(z,x+L') = f\dd'(z,x), && g\vv''(z,x+L') = f\vv'(z,x)~
\end{align*} for $x\in L$, $z\in Z$. To check that these four maps are well-defined, one computes \begin{align*}
    &f\dd''(x\dashv y+L',z) = f\dd'(x\dashv y,z) \overset{\text{C4}}{=} f\dd'(x,y\dashv z) = 0, \\& f\dd''(x\vdash y+L',z) = f\dd'(x\vdash y,z) \overset{\text{C2}}{=} f\vv'(x,y\dashv z) = 0, \\~ \\ &g\dd''(z,x\dashv y+L') = f\dd'(z,x\dashv y) \overset{\text{C4}}{=} f\dd'(z\dashv x,y) = 0, \\ &g\dd''(z,x\vdash y+L') = f\dd'(z,x\vdash y) \overset{\text{C1}}{=} f\dd'(z,x\dashv y) \overset{\text{C4}}{=} f\dd'(z\dashv x,y) = 0, \\~ \\ &f\vv''(x\dashv y+L',z) = f\vv'(x\dashv y,z) \overset{\text{C3}}{=} f\vv'(x\vdash y, z) \overset{\text{C5}}{=} f\vv'(x,y\vdash z) = 0, \\ & f\vv''(x\vdash y+L',z) = f\vv'(x\vdash y,z) \overset{\text{C5}}{=} f\vv'(x,y\vdash z) = 0, \\~ \\ & g\vv''(z,x\dashv y+L') = f\vv'(z,x\dashv y) \overset{\text{C2}}{=} f\dd'(z\vdash x,y) = 0, \\ & g\vv''(z,x\vdash y+L') = f\vv'(z,x\vdash y) \overset{\text{C5}}{=} f\vv'(z\vdash x, y) = 0
\end{align*} since $z\in Z(L)$. Hence \begin{align*}
    (f\dd'',g\dd'',f\vv'',g\vv'') &\in (\Bil(L/L'\times Z, \F)\oplus \Bil(Z\times L/L', \F))^2 \\ &\cong (L/L'\otimes Z \oplus Z\otimes L/L')^2.
\end{align*} Now let $(f\dd',f\vv')\in \cB^2(L,\F)$. Then there exists a linear transformation $\E:L\xrightarrow{} \F$ such that $f\dd'(x,y) = -\E(x\dashv y)$ and $f\vv'(x,y) = -\E(x\vdash y)$ for all $x,y\in L$. One computes \begin{align*}
    f\dd''(x+L',z) = f\dd'(x,z) = -\E(x\dashv z) = 0, && f\vv''(x+L',z) = f\vv'(x,z) = -\E(x\vdash z) = 0, \\ g\dd''(z,x+L') = f\dd'(z,x) = -\E(z\dashv x) = 0, && g\vv''(z,x+L') = f\vv'(z,x) = -\E(z\vdash x) = 0~
\end{align*} since $z\in Z(L)$. Hence, a map $\delta:(f\dd',f\vv')+\cB^2(L,\F)\mapsto (f\dd'',g\dd'',f\vv'',g\vv'')$ is induced that is clearly linear since $f\dd'$, $f\vv'$, $f\dd''$, $g\dd''$, $f\vv''$, and $g\vv''$ are all in vector spaces of bilinear forms and the latter four are defined by the first two.

\begin{thm}\label{dias batten 4.1}
Let $Z$ be a central ideal of a diassociative algebra $L$. The sequence \[\cH^2(L/Z,\F)\xrightarrow{\Inf} \cH^2(L,\F)\xrightarrow{\delta} (L/L'\otimes Z \oplus Z\otimes L/L')^2\] is exact.
\end{thm}

\begin{proof}
Let $(f\dd,f\vv)\in \cZ^2(L/Z,\F)$. Then $\Inf((f\dd,f\vv) + \cB^2(L/Z, \F)) = (f\dd',f\vv') + \cB^2(L,\F)$ where $f\dd'(x,y) = f\dd(x+Z,y+Z)$ and $f\vv'(x,y) = f\vv(x+Z,y+Z)$ for $x,y\in L$. Moreover, \[\delta((f\dd',f\vv')+\cB^2(L,\F)) = (f\dd'',g\dd'',f\vv'',g\vv'')\] where
\begin{align*}
    f\dd''(x+L',z) = f\dd'(x,z) = f\dd(x+Z,z+Z) &= 0, \\
    g\dd''(z,x+L') = f\dd'(z,x) = f\dd(z+Z,x+Z) &= 0, \\
    f\vv''(x+L',z) = f\vv'(x,z) = f\vv(x+Z,z+Z) &= 0, \\
    g\vv''(z,x+L') = f\vv'(z,x) = f\vv(z+Z,x+Z) &= 0
\end{align*} for all $x\in L$ and $z\in Z$. Thus \begin{align*}
    \delta(\Inf((f\dd,f\vv) + \cB^2(L/Z,\F))) &= \delta((f\dd',f\vv')+\cB^2(L,\F)) \\ &= (f\dd'',g\dd'',f\vv'',g\vv'') \\ &= (0,0,0,0)
\end{align*} which implies that $\ima(\Inf)\subseteq \ker \delta$.

Conversely, suppose $\delta((f\dd',f\vv')+\cB^2(L,\F)) = (f\dd'',g\dd'',f\vv'',g\vv'') = (0,0,0,0)$ for some cocycle $(f\dd',f\vv')\in \cZ^2(L,\F)$. In other words, \begin{align*}
    0 = f\dd''(x+L',z) = f\dd'(x,z), && 0 = f\vv''(x+L',z) = f\vv'(x,z),\\
    0 = g\dd''(z,x+L') = f\dd'(z,x), && 0 = g\vv''(z,x+L') = f\vv'(z,x)~
\end{align*} for all $x\in L$, $x\in Z$. Hence \begin{align*}
    f\dd'(x+z,y+z') = f\dd'(x,y) + f\dd'(x,z') + f\dd'(z,y) + f\dd'(z,z') = f\dd'(x,y),\\ f\vv'(x+z,y+z') = f\vv'(x,y) + f\vv'(x,z') + f\vv'(z,y) + f\vv'(z,z') = f\vv'(x,y)~
\end{align*} for all $z,z'\in Z$, which implies that the bilinear forms $g\dd:L/Z\times L/Z\xrightarrow{} \F$ and $g\vv:L/Z\times L/Z\xrightarrow{} \F$, defined by $g\dd(x+Z,y+Z) = f\dd'(x,y)$ and $g\vv(x+Z,y+Z) = f\vv'(x,y)$, are well-defined. Furthermore, $(g\dd,g\vv)\in \cZ^2(L/Z,\F)$ since $(f\dd',f\vv')$ is a cocycle. Thus \[\Inf((g\dd,g\vv) + \cB^2(L/Z,\F)) = (f\dd',f\vv')+\cB^2(L,\F),\] which implies that $\ker \delta \subseteq \ima(\Inf)$.
\end{proof}

One of our criteria will involve this $\delta$ map. Another will involve the \textit{natural map} $\beta$ that appears in the following analogue of the Ganea sequence.

\begin{thm}\label{dias batten 4.2}
(Ganea Sequence) Let $Z$ be a central ideal of a finite-dimensional diassociative algebra $L$. Then the sequence \[(L/L'\otimes Z \oplus Z\otimes L/L')^2 \xrightarrow{} M(L) \xrightarrow{} M(L/Z)\xrightarrow{} L'\cap Z\xrightarrow{} 0\] is exact.
\end{thm}

\begin{proof}
Let $F$ be a free diassociative algebra such that $L=F/R$ and $Z=T/R$ for ideals $T$ and $R$ of $F$. Since $Z\subseteq Z(L)$, we have $T/R\subseteq Z(F/R)$, and thus $F\lozenge T + T\lozenge F\subseteq R$. Now inclusion maps $\overline{\beta}:R\cap F'\xrightarrow{} T\cap F'$ and $\overline{\gamma}:T\cap F'\xrightarrow{} T\cap (F'+R)$ induce homomorphisms \[\frac{R\cap F'}{\FR}\xrightarrow{\beta} \frac{T\cap F'}{F\lozenge T+T\lozenge F} \xrightarrow{\gamma} \frac{T\cap (F'+R)}{R}\xrightarrow{} 0.\] Since $R\subseteq T$, \[\frac{T\cap (F'+R)}{R} = \frac{(T+R)\cap (F'+R)}{R} \cong \frac{(T\cap F')+R}{R}\] which implies that $\gamma$ is surjective. By Theorem \ref{dias batten 1.12}, \[M(L)\cong \frac{R\cap F'}{\FR} \hspace{.75cm} \text{ and }\hspace{.75cm} M(L/Z)\cong \frac{T\cap F'}{F\lozenge T + T\lozenge F}.\] Also \[L'\cap Z\cong (F/R)'\cap (T/R)\cong \frac{F'+R}{R}\cap \frac{T}{R} \cong \frac{(F'+R)\cap T}{R}.\] Therefore the sequence $M(L/Z)\xrightarrow{\gamma} L'\cap Z\xrightarrow{} 0$ is exact. Since \[\ker \gamma = \frac{(T\cap F')\cap R}{F\lozenge T + T\lozenge F} = \frac{R\cap F'}{F\lozenge T + T\lozenge F} = \ima \beta,\] the sequence $M(L)\xrightarrow{\beta} M(L/Z)\xrightarrow{\gamma} L'\cap Z$ is exact.

It remains to show that $(L/L'\otimes Z \oplus Z\otimes L/L')^2 \xrightarrow{} M(L) \xrightarrow{\beta} M(L/Z)\xrightarrow{} L'\cap Z$ is exact. Define four maps \begin{align*}
    \theta\dd:\frac{T}{R}\times \frac{F}{R+F'}\xrightarrow{} \frac{R\cap F'}{\FR}, && \theta\vv:\frac{T}{R}\times \frac{F}{R+F'}\xrightarrow{} \frac{R\cap F'}{\FR},\\ \alpha\dd:\frac{F}{R+F'}\times, \frac{T}{R}\xrightarrow{} \frac{R\cap F'}{\FR}, && \alpha\vv:\frac{F}{R+F'}\times \frac{T}{R}\xrightarrow{} \frac{R\cap F'}{\FR}~
\end{align*} by \begin{align*}
    \theta\dd(\overline{t},\overline{f}) = t\dashv f+ (\FR), && \theta\vv(\overline{t},\overline{f}) = t\vdash f + (\FR), \\ \alpha\dd(\overline{f},\overline{t}) = f\dashv t + (\FR), && \alpha\vv(\overline{f},\overline{t}) = f\vdash t+ (\FR)~
\end{align*} for $t\in T$, $f\in F$. These maps are bilinear since multiplication operations are bilinear. To check that they are well-defined, suppose $(t+R, f+(R+F')) = (t'+R, f'+(R+F'))$ for $t,t'\in T$ and $f,f'\in F$. Then $t-t'\in R$ and $f-f'\in R+F'$, which implies that $t=t'+r$ and $f=f'+x$ for some $r\in R$ and $x\in R+F'$. One computes \begin{align*}
    t\dashv f - t'\dashv f' &= (t'+r)\dashv (f'+x) - t'\dashv f' \\ &= r\dashv f' + r\dashv x + t'\dashv x \\ &\in (R\dashv F) + (R\dashv F) + (T\dashv R + T\dashv F')
\end{align*} which is contained in $\FR$ since \begin{align*}
    T\dashv F' &= T\dashv (F\dashv F) + T\dashv (F\vdash F) \\ &= (T\dashv F)\dashv F + T\dashv (F\dashv F) \\ &= (T\dashv F)\dashv F + (T\dashv F)\dashv F
\end{align*} and $T\dashv F\subseteq R$. Next, \begin{align*} t\vdash f - t'\vdash f' &= (t'+r)\vdash (f'+x) - t'\vdash f' \\ &= r\vdash f' + r\vdash x + t'\vdash x \\ &\in (R\vdash F) + (R\vdash F) + (T\vdash R + T\vdash F')
\end{align*} which is also contained in $\FR$ since \begin{align*}
    T\vdash F' &= T\vdash(F\dashv F) + T\vdash (F\vdash F) \\ &= (T\vdash F)\dashv F + (T\vdash F)\vdash F
\end{align*} and $T\vdash F\subseteq R$. Expressions $f\dashv t - f'\dashv t'$ and $f\vdash t - f'\vdash t'$ fall in $\FR$ by similar manipulations, i.e. via the identities of diassociative algebras and the fact that $F\dashv T$, $F\vdash T$, $T\dashv F$, and $T\vdash F$ are contained in $R$. Thus our bilinear forms $\theta\dd$, $\alpha\dd$, $\theta\vv$, and $\alpha\vv$ are well-defined, and so induce linear maps \begin{align*}
    \overline{\theta\dd}:\frac{T}{R}\otimes \frac{F}{R+F'}\xrightarrow{} \frac{R\cap F'}{\FR}, && \overline{\theta\vv}:\frac{T}{R}\otimes \frac{F}{R+F'}\xrightarrow{} \frac{R\cap F'}{\FR}, \\ \overline{\alpha\dd}:\frac{F}{R+F'}\otimes \frac{T}{R}\xrightarrow{} \frac{R\cap F'}{\FR}, && \overline{\alpha\vv}:\frac{F}{R+F'}\otimes \frac{T}{R}\xrightarrow{} \frac{R\cap F'}{\FR}.
\end{align*} These, in turn, yield a linear transformation \[\overline{\theta}:\Big(\frac{F}{R+F'}\otimes \frac{T}{R} \oplus \frac{T}{R}\otimes \frac{F}{R+F'}\Big)^2 \xrightarrow{} \frac{R\cap F'}{\FR}\] defined by $\overline{\theta}(a,b,c,d) = \overline{\alpha\dd}(a) + \overline{\theta\dd}(b) + \overline{\alpha\vv}(c) + \overline{\theta\vv}(d)$. The image of $\overline{\theta}$ is \[\frac{F\lozenge T+T\lozenge F}{\FR}\] which is precisely equal to $\{x + (\FR)~|~ x\in R\cap F',~ x\in F\lozenge T + T\lozenge F\} = \ker \beta$. Thus the final part of our sequence is exact.
\end{proof}

\begin{cor}
(Stallings Sequence) Let $Z$ be a central ideal in a finite-dimensional diassociative algebra $L$. Then the sequence \[M(L)\xrightarrow{} M(L/Z)\xrightarrow{} Z\xrightarrow{} L/L'\xrightarrow{} \frac{L}{Z+L'}\xrightarrow{} 0\] is exact.
\end{cor}

\begin{proof}
Follows by the same logic as the Leibniz case with the replacements $\FR$ for $FR+RF$ and $F\lozenge T+T\lozenge F$ for $FT+TF$.
\end{proof}

\subsection{The Main Result}
Let $Z$ be a central ideal in a finite-dimensional diassociative algebra $L$. We will prove the equivalence of the following statements: \begin{enumerate}
    \item $\delta$ is the trivial map,
    \item the natural map $\beta$ is injective,
    \item $M(L)\cong \frac{M(L/Z)}{L'\cap Z}$,
    \item $Z\subseteq Z^*(L)$.
\end{enumerate} The following two lemmas form the equivalence of our first three statements. Both hold similarly to their Leibniz analogues.

\begin{lem}
Let $Z$ be a central ideal in a finite-dimensional diassociative algebra $L$ and consider the map $\delta:M(L)\xrightarrow{} (L/L'\otimes Z\oplus Z\otimes L/L')^2$ as defined in Theorem \ref{dias batten 4.1}. Then \[M(L)\cong \frac{M(L/Z)}{L'\cap Z}\] if and only if $\delta$ is the trivial map.
\end{lem}

\begin{lem}
Let $Z$ be a central ideal in a finite-dimensional diassociative algebra $L$ and consider the natural map $\beta:M(L)\xrightarrow{} M(L/Z)$ as defined in Theorem \ref{dias batten 4.2}. Then \[M(L)\cong \frac{M(L/Z)}{L'\cap Z}\] if and only if $\beta$ is injective.
\end{lem}

It remains to show that these conditions are equivalent to $Z\subseteq Z^*(L)$. As in \cite{mainellis batten}, a central extension $0\xrightarrow{} A\xrightarrow{} B\xrightarrow{} C\xrightarrow{} 0$ is called \textit{stem} if $A\subseteq B'$. Consider a free presentation $0\xrightarrow{} R\xrightarrow{} F\xrightarrow{\pi} L\xrightarrow{} 0$ of $L$ and let $\overline{X}$ denote the quotient algebra $\frac{X}{\FR}$ for any $X$ such that $\FR\subseteq X\subseteq F$. Since $R=\ker \pi$ and $\FR\subseteq R$, $\pi$ induces a homomorphism $\overline{\pi}:\overline{F}\xrightarrow{} L$ such that the diagram \[\begin{tikzcd}
F\arrow[r,"\pi"]\arrow[d]& L\\
\overline{F} \arrow[ur, swap, "\overline{\pi}"]
\end{tikzcd}\] commutes. Since $\overline{R}\subseteq Z(\overline{F})$, there exists a complement $\frac{S}{\FR}$ to $\frac{R\cap F'}{\FR}$ in $\frac{R}{\FR}$ where $S\subseteq R\subseteq \ker \pi$ and $\overline{S}\subseteq \overline{R}\subseteq \ker\overline{\pi}$. Thus $\overline{\pi}$ induces a homomorphism $\pi_S:F/S\xrightarrow{} L$ such that the extension $0\xrightarrow{} R/S\xrightarrow{} F/S\xrightarrow{\pi_S} L\xrightarrow{} 0$ is central. This extension is stem since $R/S\cong \frac{R\cap F'}{\FR} = \ker \pi_S$ implies that $F/S$ is a cover of $L$.

\begin{lem}\label{dias center in center}
For every free presentation $0\xrightarrow{} R\xrightarrow{} F\xrightarrow{\pi} L\xrightarrow{} 0$ of $L$ and every central extension $0\xrightarrow{}\ker \U \xrightarrow{} E\xrightarrow{\U} L\xrightarrow{} 0$, one has $\overline{\pi}(Z(\overline{F}))\subseteq \U(Z(E))$.
\end{lem}

\begin{proof}
Since the identity map $\text{id}:L\xrightarrow{} L$ is a homomorphism, we can invoke Lemma \ref{dias restriction to R}, yielding a homomorphism $\beta:\overline{F}\xrightarrow{} E$ such that the diagram \[\begin{tikzcd}
0\arrow[r]& \frac{R}{\FR}\arrow[r] \arrow[d, "\gamma"] & \frac{F}{\FR} \arrow[r, "\overline{\pi}"] \arrow[d,"\beta"] & L\arrow[r] \arrow[d, "\text{id}"] &0 \\
0\arrow[r] &\ker \U \arrow[r] & E\arrow[r, "\U"] &L\arrow[r] &0
\end{tikzcd}\] is commutative (where $\gamma$ is the restriction of $\beta$ to $\overline{R}$). As in the Leibniz case, we obtain $E=A+\beta(\overline{F})$ where $A=\ker \U$. It remains to show that $\beta(Z(\overline{F}))$ centralizes $A$ and $\beta(\overline{F})$. We first compute $\beta(Z(\overline{F}))\ast\beta(\overline{F}) = \beta(Z(\overline{F})\ast\overline{F}) = \beta(0) = 0$ and $\beta(\overline{F})\ast\beta(Z(\overline{F})) = \beta(\overline{F}\ast Z(\overline{F})) = \beta(0) = 0$, where $\ast$ ranges over $\dashv$ and $\vdash$. Next, we know that $A\ast E$ and $E\ast A$ are zero, and so $A\ast\beta(Z(\overline{F}))$ and $\beta(Z(\overline{F}))\ast A$ must be zero as well. Thus $\beta(Z(\overline{F}))$ centralizes $E$, which implies that $\beta(Z(\overline{F}))\subseteq Z(E)$. Therefore $\U\circ \beta(Z(\overline{F}))\subseteq \U(Z(E))$, or $\overline{\pi}(Z(\overline{F}))\subseteq \U(Z(E))$.
\end{proof}

The rest of the paper follows similarly to its Leibniz case via with simple replacements. By applying the preceding results and above discussion (concerning $F/S$ as a cover), one thereby obtains the main result from Chapter 4 of \cite{batten} for the diassociative setting.

\begin{thm}\label{dias batten 4.7}
For every free presentation $0\xrightarrow{} R\xrightarrow{} F\xrightarrow{\pi} L\xrightarrow{} 0$ of $L$ and every stem extension $0\xrightarrow{} \ker \U\xrightarrow{} E\xrightarrow{\U} L\xrightarrow{} 0$, one has $Z^*(L) = \overline{\pi}(Z(\overline{F})) = \U(Z(E))$.
\end{thm}

\begin{lem}
Let $Z$ be a central ideal of finite-dimensional diassociative algebra $L$ and let \[\beta:M(L)\xrightarrow{} M(L/Z)\] be as in Theorem \ref{dias batten 4.2}. Then $Z\subseteq Z^*(L)$ if and only if $\beta$ is injective.
\end{lem}

\begin{thm}\label{dias batten 4.9}
Let $Z$ be a central ideal of a finite-dimensional diassociative algebra $L$ and let \[\delta:M(L)\xrightarrow{} (L/L'\otimes Z\oplus Z\otimes L/L')^2\] be as in Theorem \ref{dias batten 4.1}. Then the following are equivalent:
\begin{enumerate}
    \item $\delta$ is the trivial map,
    \item the natural map $\beta$ is injective,
    \item $M(L)\cong \frac{M(L/Z)}{L'\cap Z}$,
    \item $Z\subseteq Z^*(L)$.
\end{enumerate}
\end{thm}

We conclude this section by narrowing our focus to when the conditions of Theorem \ref{dias batten 4.9} hold for $Z=Z(L)$. Under this assumption, the center of the cover goes onto the center of the algebra.

\begin{thm}
Let $L$ be a diassociative algebra and $Z(L)$ be the center of $L$. If $Z(L)\subseteq Z^*(L)$, then $\U(Z(E)) = Z(L)$ for every stem extension $0\xrightarrow{} \ker \U\xrightarrow{} E\xrightarrow{\U} L\xrightarrow{} 0$.
\end{thm}

\section*{Acknowledgements}
The author would like to thank Ernest Stitzinger for the many helpful discussions.

\end{document}